\documentclass[11pt]{amsart}
\usepackage[latin1]{inputenc}
\usepackage{graphicx,color,verbatim}
\usepackage{amscd,amsmath,amsfonts,amssymb}
\usepackage{mathrsfs}
\newtheorem{theorem}{Theorem}[section]
\newtheorem{lemma}[theorem]{Lemma}

\newtheorem{remark}[theorem]{Remark}
\newtheorem{problem}[theorem]{Problem}

\newtheorem{claim}{Claim}[section]

\numberwithin{equation}{section}
\newcommand{\R}{\mathbb R}
\newcommand{\N}{\mathbb N}
\newcommand{\M}{\mathcal M}
\newcommand{\Sp}{\mathbb S}
\newcommand{\essinf}{\mathop{\rm essinf}}

\begin{document} 
\title[Existence results for some problems on Riemannian manifolds]{Existence results for some problems\\ on Riemannian manifolds}
\author{Giovanni Molica Bisci}
\address[G. Molica Bisci]{Dipartimento P.A.U. \\
UniversitÃ  degli Studi Mediterranea di Reggio Calabria\\
Salita Melissari - Feo di Vito\\
89124 Reggio Calabria, Italy}
\email{\tt gmolica@unirc.it}
\author{Du\v{s}an D. Repov\v{s}}
\address[D.D. Repov\v{s}]{Faculty of Education, and Faculty of Mathematics and Physics\\
University of Ljubljana \&
Institute of Mathematics, Physics and Mechanics, 1000 Ljubljana, Slovenia}
\email{\tt dusan.repovs@guest.arnes.si}
\author{Luca Vilasi}
\address[L. Vilasi]{Department of Mathematical and Computer Sciences, Physical Sciences and Earth Sciences\\
University of Messina\\
Viale F. Stagno d'Alcontres, 31\\
98166 Messina, Italy}
\email{\tt lvilasi@unime.it}
\begin{abstract}
By using variational techniques we provide new existence results for Yamabe-type equations with subcritical perturbations set on a compact $d$-dimensional ($d\geq 3$) Riemannian manifold without boundary. As a direct consequence of our main theorems, we prove the existence of at least one solution to the following singular Yamabe-type problem
$$
\left\lbrace
\begin{array}{ll}
-\Delta_g w + \alpha(\sigma)w  = \mu K(\sigma) w^\frac{d+2}{d-2} +\lambda \left( w^{r-1} + f(w)\right),  \quad \sigma\in\M &\\
&\\
w\in H^2_\alpha(\M), \quad  w>0 \mbox{ in } \mathcal M &
\end{array}
\right.
$$
where, as usual, $\Delta_g$ denotes the Laplace-Beltrami operator on $(\M,g)$, $\alpha, K:\M\to\R$ are positive (essentially) bounded functions, $r\in(0,1)$, and $f:[0,+\infty)\to[0,+\infty)$ is a subcritical continuous function. Restricting ourselves to the unit sphere ${\mathbb{S}}^d$
via the
stereographic projection, we also solve some parametrized Emden-Fowler equations in the Euclidean case.
\end{abstract} 
\maketitle
\section{Introduction}\label{sec:introduzione}
In the present paper we explore the existence of solutions to the
following  problem
\begin{equation}\tag{$P_{\lambda,\mu}$}\label{problema}
-\Delta_g w + \alpha(\sigma)w =\mu K(\sigma)|w|^\frac{4}{d-2}w \!+\! \lambda f(w), \quad \sigma\!\in\!\M, \; w\!\in\! H^2_1(\M),
\end{equation}
where $(\M,g)$ is a compact $d$-dimensional Riemannian manifold without boundary, $d\geq 3$,
 $\Delta_g$ is the Laplace-Beltrami operator on $(\M,g)$, expressed in local coordinates $(x_1,\ldots,x_d)$ by
$$
\Delta_g w := g^{ij}\left(\frac{\partial^2 w}{\partial x^i\partial x^j} -\Gamma_{ij}^k\frac{\partial w}{\partial x^k} \right)
$$
where $\Gamma_{ij}^k$ are the well-known Christoffel
symbols,
 the functions $\alpha, K:\M\to\R$ belong to the class
\begin{equation}\label{classeLambdapiu}
\Lambda_+(\M):=\left\lbrace \varphi\in L^\infty(\M,\R):\essinf_{\sigma\in\M} \varphi(\sigma)>0\right\rbrace,
\end{equation}
$\lambda,\mu$ are positive real parameters, and $f:\R\to\R$ is a continuous function with subcritical growth. 

Problem \eqref{problema} displays the same structure as the Yamabe equation with the addition of a subcritical perturbation term. As is well-known, the latter equation is intimately related to the following problem arising in differential geometry.

\begin{problem}
Given a $d$-dimensional compact Riemannian manifold $(M,g)$, $d>2$, with scalar curvature $k=k(\sigma)$, find a metric $\tilde g$ conformal to $g$, with constant scalar curvature $\tilde k$.
\end{problem}

\noindent Setting $\tilde g := w^{4/(d-2)} g$, where $w>0$ is the conformal factor, this requirement can be stated in terms of partial differential equations as follows:\\

\emph{Find $w\in C^\infty(\M)$, $w>0$, satisfying
\begin{equation}\label{yamabe}
-\Delta_g w +\frac{d-2}{4(d-1)}k(\sigma) w = \tilde k w^\frac{d+2}{d-2} \quad \text{\rm in }\M.
\end{equation}}

The history of equation \eqref{yamabe} and attempts on its resolution spans more than two decades. In 1960, Yamabe \cite{yam1960on} claimed to have found a solution but his proof turned out to be wrong and was repaired by Trudinger \cite{tru1968remarks} for the case when the conformal class of the reference metric is nonpositive. Some years later, Aubin \cite{aub1976equations} provided a positive answer to the question for all manifolds of dimension $d\geq 6$ which are not conformally flat; the remaining  and more difficult cases were definitively solved by Schoen \cite{sch1984conformal}.

Elliptic problems set on compact Riemannian manifolds without boundary, as well as their applications to Emden-Fowler equations, have been the object of recent study, especially in the sublinear case. In \cite{krirad2009sublinear} Krist\'{a}ly and R\u{a}dulescu obtained multiple solutions to the problem
\begin{equation}\tag{$P_{\mu}$} \label{problemapmu}
-\Delta_{g} w + \alpha (\sigma) w= \mu
K(\sigma)f(w),\quad \sigma\in {\mathcal{M}} ,\ w\in H_1^2({\mathcal{M}}),
\end{equation}
when $f$ has a sublinear behaviour at infinity. More precisely, in their Theorem 1.1 they proved the existence of two nontrivial solutions to \eqref{problemapmu}, for $\mu$ sufficiently large, through a careful analysis of the mountain pass geometry of the functional energy. Under some natural compactness assumptions, a similar approach applies to problem \eqref{problema} as well, and leads to the existence of at least one nontrivial solution, as shown in our Theorem \ref{secondoteorema} (see Section \ref{mountainpass}).

However, here we are going to propose a different approach, quite new in the framework of elliptic equations on manifolds, which in particular does not use the celebrated Lions concentration-compactness principle (see \cite{lio1985the}) or one of its many variants. As it is well-known, the latter represents a major tool in the study of critical elliptic equations and has been intensively used in literature in many different contexts.

Our first existence result (Theorem \ref{principale}) relies upon direct minimization techniques on small balls of the energy space. To be more precise, we prove that, for every $\mu>0$, the restriction of the functional
\begin{align*}
	w &\mapsto \frac{1}{2}\int_{\M} \left( |\nabla w(\sigma)|^2 + \alpha(\sigma)w(\sigma)^2\right)  d\sigma_g\\
	&\quad -\frac{\mu}{2^*}\int_\M K(\sigma)|w(\sigma)|^{2^*}d\sigma_g, \quad w\in H_\alpha^2(\M),
\end{align*}
to a ball of $H_\alpha^2(\M)$ centered at $0$ and of small enough radius, is sequentially weakly lower semicontinuous (see Lemma \ref{lemmasemicontinuita}).

 Consequently, the whole energy functional associated with \eqref{problema} is locally
sequentially weakly lower semicontinuous and so it will admit, for any $\mu>0$ and sufficiently small $\lambda$, a local minimum. We point out that throughout our paper the nonlinearity $f$ is merely required to be subcritical and with an appropriate asymptotic behaviour at $0$. The key point is the smallness of $\lambda$, which needs to belong to the range of a suitable rational function, involving among other things the sharp constant of the continuous (but not compact) embedding $H_\alpha^2(\M)\hookrightarrow L^{2d/(d-2)}(\M)$, see Remark \ref{costanti}.

We furthermore emphasize that our approach, originally due to Cha\-bro\-wski \cite{cha1995on}, has been previously used in literature for studying classical quasilinear $p$-Laplacian equations involving critical nonlinearities (see \cite{farfar2015a}, by
 which this paper was inspired, and \cite{squ2004two}). For the sake of completeness, we also mention \cite{molvil2020isometry}, where the noncompact
 analogue
  of \eqref{problema} was investigated, and the recent papers \cite{hajmolvil2016existence,mawmol2017a,molrep2017yamabe} where, by adopting similar variational methods, the existence of weak solutions for nonlocal fractional equations and subelliptic problems on Carnot groups were studied in situations of lack of compactness.

Theorem \ref{principale} has connections with several results obtained in the case of bounded domains of $\R^d$, like the main one of \cite{tar1992on}, where Tarantello showed that the problem
\begin{equation}\label{Tarantello1}
\begin{cases}
-\Delta u=|u|^\frac{4}{d-2} u +h(x) \quad \mbox{ in } \Omega\\
u=0 \quad \mbox{ on } \partial \Omega,
\end{cases}
\end{equation}
where $\Omega\subset\R^d$, $d>2$, and $h\in H^{-1}(\Omega)$, admits two distinct weak solutions provided that $h$ has a sufficiently small $\|\cdot\|_{-1,2}$-norm, namely
$$
\|h\|_{-1,2}\leq \frac{4}{d-2}\left(\frac{d-2}{d+2}\right)^{(d+2)/4}S^{d/4},
$$
where $S$ is the best critical Sobolev constant for the
embedding $H^1_0(\Omega)\hookrightarrow L^{2^*}(\Omega)$.

 There are other connections with the central result of \cite{gazruf1997lower}, dealing with the perturbed critical elliptic problem
\begin{equation*}\label{GaRu22}
\begin{cases}
-\Delta u=|u|^\frac{4}{d-2} u + g(x,u) \quad {\mbox{ in }} \Omega\\
u=0 \quad {\mbox{ on }} \partial \Omega,
\end{cases}
\end{equation*}
where $\Omega\subset\R^d$, $d\geq 3$, and the nonlinearity $g$ has a subcritical growth at infinity.

Among all cases covered by \eqref{problema}, a remarkable one is given by
\begin{align}\tag{$\widetilde P_{\lambda,\mu}$}\label{casosfera}
	&\quad -\Delta_{h}w + s(1-s-d)w \\
	\notag &\qquad= \mu K(\sigma)|w|^\frac{4}{d-2}w + \lambda f(w),\quad \sigma\in {\mathbb{S}}^d,\, w\in H_1^2({\mathbb{S}}^d),
\end{align}
where ${\mathbb{S}}^d$ is the unit sphere in $\R^{d+1}$, $h$ is the standard metric induced by the embedding ${\mathbb{S}}^d\hookrightarrow \R^{d+1}$, $s$ is a constant such that $1-d<s<0$ and $\Delta_h$ denotes the Laplace-Beltrami operator on $({\mathbb{S}}^d,h)$. It yields equation \eqref{yamabe} on $\Sp^d$ when $s=-d/2$ or $s=-d/2+1$ (see for instance \cite{aub1998some,cotili1991equations,cotili1995equations,heb2000nonlinear,kazwar1975scalar,vazver1991solutions} and references therein for a wider framework on these topics).

 Problem \eqref{casosfera} has an analogue in the Euclidean context: any existence result for \eqref{casosfera} indeed can be translated, by an appropriate change of coordinates, into an existence result for the following parametrized Emden-Fowler type equation,
\begin{align}\tag{$\overline P_{\lambda,\mu}$}\label{emdenfowler}
	-\Delta u  &=\mu |x|^\frac{2(2-2s-d)}{d-2}K\left(\frac{x}{|x|}\right)|u|^\frac{4}{d-2} u\\ \notag &\quad + \lambda |x|^{s-2} f\left(\frac{u}{|x|^s}\right) , \quad x\in \R^{d+1}\setminus\{0\},
\end{align}
see Section \ref{SectionEF} for details. The latter represents a generalization of the more common (and unperturbed) equation
\begin{equation}\tag{$\overline P_\mu$}\label{emdenfowlernonperturbata}
-\Delta u =
\mu|x|^{s-2}K\left( \frac{x}{|x|}\right) f\left(\frac{u}{|x|^s}\right), \quad x\in \R
^{d+1}\setminus \{0\},
\end{equation}
which, in the presence of a pure superlinear nonlinearity (that is, for $f(t)=|t|^{p-1}t$, $p>1$), has been treated by minimization and minimax techniques in the classical papers \cite{cotili1995equations} and \cite{vazver1991solutions}.

If the perturbation term in \eqref{problema} is modified with the addition of a singular power at zero, then via the same approach as before
(cf. also \cite{farfar2015a}), it is still possible to get nontrivial solutions to the singular Yamabe problem
\begin{equation*}\label{problemasingolare}
\begin{cases}
-\Delta_g w + \alpha(\sigma)w  = \mu K(\sigma) w^\frac{d+2}{d-2} +\lambda \left( w^{r-1} + f(w)\right),  \quad \sigma\in\M \\
w\in H^2_\alpha(\M), \quad  w>0 \mbox{ in } \mathcal M,
\end{cases}
\end{equation*}
where $r\in(0,1)$ and $f:[0,+\infty)\to[0,+\infty)$ is continuous and subcritical, see Theorem \ref{princsingolare}.\par
 As a consequence, it is straightforward to deduce the existence of at least one nontrivial solution to the following singular equation
\begin{equation*}
\begin{cases}
-\Delta_h w + s(1-s-d)w  = \mu K(\sigma) w^\frac{d+2}{d-2} +\lambda \left( w^{r-1} + f(w)\right),  \quad \sigma\in \Sp^d \\
w\in H^2_1(\Sp^d), \quad  w>0 \mbox{ in } \Sp^d,
\end{cases}
\end{equation*}
and this is the content of our Theorem \ref{Th11}.\par
For the sake of completeness we cite related paper \cite{filpucrob2009on} in which the authors studied the existence of nonnegative solutions for the doubly critical equation of the form
$$
-\Delta_p u-\mu\frac{u^{p-1}}{|x|^p}=u^{p^*-1}+\frac{u^{p^*(s)-1}}{|x|^s},\quad {\rm in}\,\,\R^d
$$
where $\Delta_p$ is the usual $p$-Laplace operator (with $1<p<d$ and $d\geq 2$), $\mu$ is a real parameter, $p^*$ is the critical Sobolev exponent, $0<s<p$ and $p^*(s):=p(d-s)/(d-p)$. We also point out that a Lane-Emden-Fowler equation on a bounded Euclidean domain and involving a singular potential was studied in \cite{dupgherad2007lane}. Furthermore, under appropriate spectral assumptions, two existence results for positive solutions of Lichnerowicz-type equations on complete (noncompact) manifolds were proved in \cite{albrig2016lichnerowicz}.\par
The paper is organized as follows. In Section \ref{preliminaries} we review some basic definitions and facts on Sobolev spaces defined on compact Riemannian manifolds, while in Section \ref{directminimization} we state and prove our main result of existence of solutions for \eqref{problema}. Section \ref{mountainpass} is devoted to another existence result, obtained this time as an application of a version of the Mountain Pass Theorem without the Palais-Smale condition due to Brezis and Nirenberg (\cite[Theorem 2.2]{brenir1983positive}) proving an existence result for problem \eqref{problema}.
 Some concrete applications of our general theorems to Emden-Fowler equations are illustrated in the last section.
 
\section{Preliminaries and variational framework}\label{preliminaries}
In this section we recall some notions and basic facts on Sobolev spaces on compact Riemannian manifolds which will be helpful in the sequel. We refer the reader to \cite{aub1998some,heb2000nonlinear} for detailed derivations of the geometric quantities, their motivation and further applications. We also mention recent monograph \cite{kriradvar2010variational} as a general reference on this subject.

Let $(\M,g)$ be a smooth compact $d$-dimensional Riemannian manifold, $d\geq 3$, without boundary and let $g_{ij}$ be the components of the metric $g$. If $\alpha\in\Lambda_+(\M)$ (see \eqref{classeLambdapiu}) and  $C^\infty(\M)$ denotes as usual the space of smooth functions defined on $\M$, we set
\begin{equation}\label{normaHduealpha}
\left\|w\right\|_{H^2_\alpha}:=\left(\int_{\M}  |\nabla w(\sigma)|^2 d\sigma_g+ \int_{\M}\alpha(\sigma)w(\sigma)^2  d\sigma_g\right)^{1/2},
\end{equation}
for every $w\in C^\infty(\M)$, where $\nabla w$ is the covariant derivative of $w$ and $d\sigma_g$ is the Riemannian measure on $\M$. In terms of local coordinates $(x^1,\ldots, x^d)$, $\nabla w$ can be represented by
$$
(\nabla^2 w)_{ij}=\frac{\partial^2 w}{\partial x^i \partial x^j} - \Gamma_{ij}^k \frac{\partial w}{\partial x^k}
$$
where
$$
\Gamma_{ij}^k:=\frac{1}{2}\left(\frac{\partial g_{lj}}{\partial x^i} + \frac{\partial g_{li}}{\partial x^j} - \frac{\partial g_{ij}}{\partial x^k}\right) g^{lk}
$$
are the usual Christoffel symbols and $g^{lk}$ are the elements of the inverse matrix of $g$, Einstein's summation convention being tacitly adopted.\par
 The Riemannian volume element $d\sigma_g$ in \eqref{normaHduealpha} is given by $d\sigma_g=\linebreak (det \, g)^{1/2}\;dx$, where $dx$ stands for the
 Lebesgue
  volume element of $\R^d$, and we set
$$
Vol_g(\M):=\int_{\M}d\sigma_g.
$$
In the special case $(\M,g)=(\Sp^d,h)$, where $\Sp^d$ is the unit sphere in $\R^{d+1}$ and $h$ is the standard metric induced by the embedding $\Sp^d\hookrightarrow\R^{d+1}$, we use the notation
$$
\omega_d:=Vol_h(\Sp^d)=\int_{\Sp^d} d\sigma_h.
$$

The Sobolev space $H_\alpha^2(\M)$ is defined as the completion of $C^\infty(\M)$ with respect to the norm $\left\| \cdot\right\|_{H_\alpha^2}$. Such a space turns out to be a Hilbert space when endowed with the inner product
\begin{align}\label{prodscalareHduealpha}
	\left\langle w_1,w_2\right\rangle_{H_\alpha^2}&=
\int_{\M}\left\langle \nabla w_1(\sigma),\nabla w_2(\sigma) \right\rangle_g  d\sigma_g \\
	\notag &\quad +  \int_{\M} \alpha(\sigma) \left\langle w_1(\sigma), w_2(\sigma)\right\rangle_g d\sigma_g,
\end{align}
for every $w_1,w_2\in H_\alpha^2(\M)$, where $\left\langle \cdot,\cdot\right\rangle_g$ is the inner product on covariant tensor fields associated with $g$. 

Because of the positivity of $\alpha$, $\left\| \cdot\right\|_{H_\alpha^2}$ is equivalent to the standard norm
\begin{equation}
\left\|w\right\|_{H^2_1}=\left(\int_{\M} |\nabla w(\sigma)|^2d\sigma_g + \int_{\M}w(\sigma)^2 d\sigma_g\right)^{1/2} ,
\end{equation}
as is plainly deducible from the following inequalities
\begin{equation}\label{equivalence}
\min\left\lbrace 1,\essinf_{\sigma\in\M}\alpha(\sigma)^{1/2}\right\rbrace \left\| w\right\|_{H_1^2}  \leq \left\| w\right\|_{H_\alpha^2} \leq \max\left\lbrace 1,\left\| \alpha\right\|_\infty^{1/2}\right\rbrace \left\| w\right\|_{H_1^2},
\end{equation}
which hold
 for any $w\in H_\alpha^2(\M)$. Hereafter, we shall drop the subscripts in \eqref{normaHduealpha} and \eqref{prodscalareHduealpha} and denote $\left\| \cdot\right\|_{H_\alpha^2}$ and $\left\langle \cdot,\cdot\right\rangle_{H_\alpha^2}$ simply by $\left\| \cdot\right\|$ and $\left\langle \cdot,\cdot\right\rangle$, respectively.\par
On account of Rellich-Kondrachov's theorem for compact manifolds without boundary, one has
\begin{equation}\label{immersione}
H_\alpha^2(\M)\hookrightarrow L^q(\M) \quad \mbox{ for every } q\in[1,2^*],
\end{equation}
where
$$
2^*=\frac{2d}{d-2},
$$
and the above embedding is also compact whenever $q\in[1,2^*)$. For any $q\in[1,2^*]$ we denote by $\left\| \cdot\right\|_q$ the standard $L^q$-norm on $\M$ and by $c_q$ the best constant of \eqref{immersione}, i.e.
$$
c_q:=\sup_{w\in H_\alpha^2(\M)\setminus\{0\}}\frac{\left\| w\right\|_q}{\left\| w\right\|} >0,
$$
while we reserve the symbol $S$ for the positive constant
\begin{equation}\label{costanteS}
S:=\inf_{w\in H_\alpha^2(\M)\setminus\{0\}}\frac{\left\| w\right\|^2}{\left\| w\right\|_{2^*}^2}.
\end{equation}

The open (respectively, closed) ball centered at $w\in H_\alpha^2(\M)$ of radius $r>0$ will be denoted by $B(w,r)$ (respectively, $B_c(u,r)$), and the sphere $\{z\in H_\alpha^2(\M): \left\| w-z\right\|=r\}$ by $\partial B(w,r)$.

In what follows, we shall make the standing assumptions that $K\in \Lambda_+(\M)$ and $f:\R\to\R$ is locally Lipschitz continuous. It is straightforward to check that problem \eqref{problema} represents the Euler-Lagrange equation of the functional
\begin{align}\label{funzenergia}
	\mathcal E_{\lambda,\mu}(w) &:= \frac{1}{2}\left\| w\right\|^2 - \frac{\mu}{2^*} \int_\M K(\sigma)|w|^{2^*} d\sigma_g \\
	\notag &\quad - \lambda\int_\M  F(w(\sigma))d\sigma_g, \quad w\in H_\alpha^2(\M),
\end{align}
where
$$
F(t):=\int_0^t f(\tau)d\tau \quad \mbox{for all } t\in\R.
$$

The critical points $w$ of $\mathcal E_{\lambda_\mu}$ are therefore the weak solutions to \eqref{problema}, i.e. they satisfy the relationship
$$
\left\langle w,z\right\rangle = \mu \int_\M K(\sigma)|w|^\frac{4}{d-2}wz d\sigma_g + \lambda\int_\M  f(w)z d\sigma_g,
$$
for all $z\in H_\alpha^2(\M)$ and, furthermore, the regularity assumptions on $K$ and $f$ force any weak solution to be actually a classical one.

It is worth observing that, fixing $\lambda,\mu\in\R$, the constant function $w(\sigma)=c\in \R$ is a (trivial) solution to \eqref{problema} if and only if
$$
c=\frac{\mu K(\sigma)|c|^\frac{4}{d-2} c + \lambda f(c)}{\alpha(\sigma)}\quad \mbox{ in }\M.
$$
Thus, constant solutions to \eqref{problema} appear as fixed points of the real function
$$
t\mapsto \frac{\mu K(\sigma)|t|^\frac{4}{d-2} t + \lambda f(t)}{\alpha(\sigma)}, \quad t\in\R.
$$

\begin{remark}\label{costanti}
\rm{It is well-known that sharp Sobolev inequalities play a crucial role in the theory partial differential
equations, from both theoretical and applied point of view,
and there is a broad literature on this subject. In our context, when $(\mathcal{M},g)=(\mathbb{S}^{d},h)$, a concrete upper bound for the constants $c_q$, $q\in[1,2d/(d-2))$, depending on the geometry of $\Sp^d$, can be obtained as
\begin{equation}\label{inequality2}
\displaystyle c_q\leq\frac{\kappa_q}{\min\left\{1,\displaystyle\min_{\sigma\in {\mathbb{S}}^{d}}\alpha(\sigma)^{1/2}\right\}},
\end{equation}
 where
\begin{equation*}\label{K}
\kappa_q:=
\begin{cases}
\omega_d^{\frac{2-q}{2q}} & \mbox{ if }
	q\in [1,2), \\[1ex]
\displaystyle\max\left\{
\displaystyle\left(\frac{q-2}{d}\omega_d^{\frac{2-q}{q}}
\right)^{1/2}, \omega_d^{\frac{2-q}{2q}}
\right\} & \mbox{ if } q\in \left[2,\frac{2d}{d-2}\right) .\\
\end{cases}
\end{equation*}
Indeed, for every $2\leq q< 2d/(d-2)$ and $w\in H^{2}_1({\mathbb{S}}^{d})$, one has
$$
\left(\int_{{\mathbb{S}}^{d}}|w(\sigma)|^qd\sigma_h\right)^{2/q}\leq \frac{q-2}{d\omega_d^{1-2/q}}\int_{{\mathbb{S}}^{d}}|\nabla w(\sigma)|^2d\sigma_h+\frac{1}{\omega_d^{1-2/q}}\int_{{\mathbb{S}}^{d}}|w(\sigma)|^2d\sigma_h,
$$
(cf. for instance \cite{bec1993sharp} and \cite{heb2000nonlinear}) and thus
\begin{align*}
	\|w\|_q &\leq \displaystyle\max\left\{
\displaystyle\left(\frac{q-2}{d}\omega_d^{\frac{2-q}{q}}
\right)^{1/2}, \omega_d^{\frac{2-q}{2q}}
\right\}\\
&\quad\times\left(\int_{{\mathbb{S}}^{d}}|\nabla w(\sigma)|^2d\sigma_h+\int_{{\mathbb{S}}^{d}}|w(\sigma)|^2d\sigma_h\right)^{1/2},
\end{align*}
which, by \eqref{equivalence}, implies the desired estimate. On the other hand, if $q\in [1,2)$, it follows from H\"{o}lder's inequality that
$$
\|w\|_q \leq \omega_d^{\frac{2-q}{2q}}\|w\|_2, \quad \mbox{ for all } w\in L^2({\mathbb{S}}^{d})
$$
and the conclusion is achieved by taking into account that
$$
\|w\|_2 \leq \|w\|_{H^2_1({\mathbb{S}}^{d})}\leq \frac{\|w\|}{\min\left\{1,\displaystyle\min_{\sigma\in {\mathbb{S}}^{d}}\alpha(\sigma)^{1/2}\right\}},
$$
for every $w\in H^{2}_1({\mathbb{S}}^{d})$.
}
\end{remark}

\section{Existence of local minimizers}\label{directminimization}
Our first existence result for problem \eqref{problema} can be stated as follows:
\begin{theorem}\label{principale}
	Let $f:\R\to\R$ satisfy the following requirements:
	\begin{itemize}
		\item[$(f_1)$] there exist $a_1,a_2>0$ and $q\in[1,2^*)$ such that
		$$
		|f(t)|\leq a_1 + a_2|t|^{q-1} \quad  \mbox{ for all } t\in\R,
		$$
		\item[$(f_2)$] $\displaystyle\liminf_{t\to 0^+}\frac{F(t)}{t^2} =+\infty$.
	\end{itemize}
	Furthermore, for any $\mu>0$ let $l_\mu:[0,+\infty)\to\R$ be the function defined by
	\begin{equation}\label{funzionelmu}
	l_\mu(t):=\frac{t-\mu c_{2^*}^{2^*}\left\| K\right\|_\infty  t^{2^*-1}}{a_1 c_{2^*}Vol_g(\M)^\frac{2^*-1}{2^*} + a_2 c_{2^*}^q Vol_g(\M)^\frac{2^*-q}{2^*} t^{q-1}} \quad  \mbox{ for all } t\geq 0.
	\end{equation}
	Then for every $\mu>0$ there exists an open interval
	$$
	\Lambda _\mu \subseteq \left(0,\max_{[0,+\infty)}l_\mu\right)
	$$
	such that, for every $\lambda\in\Lambda_\mu$, \eqref{problema} admits a nontrivial solution $w_{0,\mu,\lambda}\in H_\alpha^2(\M)$.
\end{theorem}
As already mentioned, our strategy consists in showing that the energy \eqref{funzenergia} possesses a nontrivial minimizer in $H_\alpha^2(\M)$. Yet, the presence of the critical term prevents the direct minimization from being immediately applicable. First we consider, for all $w\in H_\alpha^2(\M)$, the functionals
$$
w\mapsto\mathcal E_{\lambda,\mu}(w) + \lambda \int_{\M}F(w)d\sigma_g, \quad w\mapsto \frac{1}{2}\left\| w\right\| ^2 - \mathcal E_{\lambda,\mu}(w),
$$
and prove some of their properties.

\begin{lemma}\label{lemmasemicontinuita}
For every $\mu>0$ there exists $\varrho_{0,\mu}>0$ such that the functional
$$
\widehat{\mathcal E}_\mu(w):=\frac{1}{2}\left\|w\right\|^2 -\frac{\mu}{2^*}\int_\M K(\sigma)|w(\sigma)|^{2^*}d\sigma_g, \quad w\in H_\alpha^2(\M),
$$
is sequentially weakly lower semicontinuous in $B_c(0,\varrho_{0,\mu})$.
\end{lemma}

\begin{proof}
Let $\mu,\varrho>0$ and let $\{w_j\}_{j\in \N}\subset B_c(0,\varrho)$ be such that $w_j\rightharpoonup w_\infty\in B_c(0,\varrho)$.  The conclusion will be achieved  by proving that
\begin{equation}\label{convergenzeFine}
\liminf_{j\rightarrow\infty}(\widehat{\mathcal E}_\mu(w_j)-\widehat{\mathcal E}_\mu(w_\infty))\geq 0.
\end{equation}

Let us observe that, for all $w_1,w_2\in H_\alpha^2(\M)$, the following basic equality holds
\begin{equation*}
\left\|w_2\right\|^2 - \left\|w_1\right\|^2 - 2\left\langle w_1, w_2-w_1\right\rangle_g = \left\|w_1-w_2\right\|^2.
\end{equation*}
Moreover, applying Br\'{e}zis-Lieb's Lemma to the sequence $\{K^{1/2^*}w_j\}_{j\in\N}\subset L^{2^*}(\M)$, one has
\begin{align}\label{GMB2}
&\quad\ \liminf_{j\rightarrow \infty}\left( \int_\M K(\sigma)|w_j(\sigma)|^{2^*}d\sigma_g - \int_\M K(\sigma)|w_\infty(\sigma)|^{2^*}d\sigma_g\right)\\
\notag &\quad=\liminf_{j\rightarrow \infty}\int_\M K(\sigma)|w_j(\sigma)-w_\infty(\sigma)|^{2^*}d\sigma_g.
\end{align}
Bearing in mind also that $w_j\rightharpoonup w_\infty$, we deduce
\begin{align*}\label{perGMB22}
&\quad\ \liminf_{j\rightarrow \infty}(\widehat{\mathcal E}_\mu(w_j)-\widehat{\mathcal E}_\mu(w_\infty))\\
\notag & = \liminf_{j\rightarrow \infty}\left( \frac{1}{2}\left(\left\|w_j\right\|^2 -\left\|w_\infty \right\|^2\right)\right.  \\
\notag &\quad \left. - \frac{\mu}{2^*}\int_{\M}K(\sigma) \left(|w_j(\sigma)|^{2^*}-|w_\infty(\sigma)|^{2^*}\right) d\sigma_g \right)\\
\notag &\geq \liminf_{j\rightarrow \infty}\left( \frac{1}{2}\left\|w_j-w_\infty\right\|^2 -\frac{\mu}{2^*}\left\| K\right\|_\infty \int_\M |w_j(\sigma)-w_\infty(\sigma)|^{2^*}d\sigma_g\right)\\
\notag & \geq \liminf_{j\rightarrow \infty} \left(\frac{1}{2}-\frac{c_{2^*}^{2^*}}{2^*}\mu\left\| K\right\|_\infty \|w_j-w_\infty\|^{2^*-2}\right)\|w_j-w_\infty\|^{2}\\
\notag & \geq \liminf_{j\rightarrow \infty}\left(\frac{1}{2} - \frac{c_{2^*}^{2^*}}{2^*}\mu\left\| K\right\|_\infty\varrho^{2^*-2}\right) \|w_j-w_\infty\|^{2}.
\end{align*}
So for
\begin{equation*}\label{rhosegnatomu}
0<\varrho\leq \bar{\varrho}_\mu:=\left(\frac{d}{(d-2) c_{2^*}^{2^*}\mu\left\| K\right\|_\infty}\right)^{\frac{d-2}{4}},
\end{equation*}
inequality \eqref{convergenzeFine} is verified and $\widehat{\mathcal E}_\mu$ is sequentially weakly lower semicontinuous in $B_c(0,\varrho_{0,\mu})$, provided that $\varrho_{0,\mu}\in (0,\bar{\varrho}_\mu)$.
\end{proof}

\begin{lemma}\label{proprietaEtilde}
Let $\lambda,\mu>0$, let $f:\R\to\R$ satisfy $(f_1)$, and let $\widetilde{\mathcal E}_{\lambda,\mu}:H_\alpha^2(\M)\to\R$ be the functional defined by
$$
\widetilde{\mathcal E}_{\lambda,\mu}(w):=\frac{\mu}{2^*}\int_\M K(\sigma) |w(\sigma)|^{2^*}d\sigma_g + \lambda \int_\M F(w)  d\sigma_g
$$
for any $w\in H_\alpha^2(\M)$. Then the following facts hold:
\begin{itemize}
\item[$(i)$] if
\begin{equation}\label{Ga1}
\limsup_{\varepsilon\rightarrow 0^+}\frac{\displaystyle\sup_{B_c(0,\varrho_0)}\widetilde{\mathcal E}_{\lambda,\mu} - \sup_{B_c(0,\varrho_0-\epsilon)}\widetilde{\mathcal E}_{\lambda,\mu}}{\varepsilon}<\varrho_0
\end{equation}
for some $\varrho_0>0$, then
\begin{equation}\label{Ga2}
\inf_{\eta<\varrho_0}\frac{\displaystyle \sup_{B_c(0,\varrho_0)}\widetilde{\mathcal E}_{\lambda,\mu} - \sup_{B_c(0,\eta)}\widetilde{\mathcal E}_{\lambda,\mu}}{\varrho_0^2-\eta^2}<\frac{1}{2};
\end{equation}
\item[$(ii)$] if \eqref{Ga2} is satisfied for some $\varrho_0>0$, then
\begin{equation}\label{VGa1}
\inf_{w\in B(0,\varrho_0)}\frac{\displaystyle \sup_{B_c(0,\varrho_0)}\widetilde{\mathcal E}_{\lambda,\mu} - \widetilde{\mathcal E}_{\lambda,\mu}(w)}{\varrho_0^2-\left\|w\right\|^2}<\frac{1}{2}.
\end{equation}
\end{itemize}
\end{lemma}

\begin{proof}
\emph{$(i)$} From the identity
$$
\frac{\displaystyle\sup_{B_c(0,\varrho_0)}\widetilde{\mathcal E}_{\lambda,\mu} - \sup_{B_c(0,\varrho_0-\epsilon)}\widetilde{\mathcal E}_{\lambda,\mu}}{\varrho_0^2 -(\varrho_0-\varepsilon)^2} =\left(  \frac{\displaystyle\sup_{B_c(0,\varrho_0)}\widetilde{\mathcal E}_{\lambda,\mu} - \sup_{B_c(0,\varrho_0-\epsilon)}\widetilde{\mathcal E}_{\lambda,\mu}}{\varepsilon}\right) \left( \frac{1}{2\varrho_0-\varepsilon}\right),
$$
it follows that
\begin{equation}\label{Ga1GG}
\limsup_{\varepsilon\rightarrow 0^+}\frac{\displaystyle\sup_{B_c(0,\varrho_0)}\widetilde{\mathcal E}_{\lambda,\mu} - \sup_{B_c(0,\varrho_0-\epsilon)}\widetilde{\mathcal E}_{\lambda,\mu}}{\varrho_0^2 -(\varrho_0-\varepsilon)^2}<\frac{1}{2}.
\end{equation}

Now, by \eqref{Ga1GG} there exists $\bar\varepsilon_0>0$ such that
$$
\frac{\displaystyle\sup_{B_c(0,\varrho_0)}\widetilde{\mathcal E}_{\lambda,\mu} - \sup_{B_c(0,\varrho_0-\varepsilon)}\widetilde{\mathcal E}_{\lambda,\mu}}{\varrho_0^2 -(\varrho_0-\varepsilon)^2}<\frac{1}{2}
$$
for every $\varepsilon\in (0,\bar\varepsilon_0)$. Setting $\eta_0:=\varrho_0-\varepsilon_0$, with $\varepsilon_0\in (0,\bar\varepsilon_0)$, we get
$$
\frac{\displaystyle\sup_{B_c(0,\varrho_0)}\widetilde{\mathcal E}_{\lambda,\mu} - \sup_{B_c(0,\eta_0)}\widetilde{\mathcal E}_{\lambda,\mu}}{\varrho_0^2 -\eta_0^2}<\frac{1}{2}
$$
and thus the conclusion follows.

\emph{$(ii)$} Thanks to inequality \eqref{Ga2}, one has
\begin{equation}\label{Ga123}
\sup_{B_c(0,\eta_0)}\widetilde{\mathcal E}_{\lambda,\mu} > \sup_{B_c(0,\varrho_0)}\widetilde{\mathcal E}_{\lambda,\mu} -\frac{1}{2}(\varrho_0^2-\eta^2_0)
\end{equation}
for some $0<\eta_0<\varrho_0$. Invoking $(f_1)$ and standard argument, $\widetilde{\mathcal E}_{\lambda,\mu}$ turns out to be weakly lower semicontinuous in $B_c(0,\eta_0)$ and therefore
\begin{equation*}\label{Ga1234}
\sup_{\partial B_c(0,\eta_0)}\widetilde{\mathcal E}_{\lambda,\mu} = \sup_{\overline{\partial B_c(0,\eta_0)}^*}\widetilde{\mathcal E}_{\lambda,\mu} = \sup_{B_c(0,\eta_0)}\widetilde{\mathcal E}_{\lambda,\mu},
\end{equation*}
where $\overline{\partial B_c(0,\eta_0)}^*$ is the weak closure of $\partial B_c(0,\eta_0)$ in $H_\alpha^2(\M)$. Therefore by \eqref{Ga123} there exists $w_0\in H_\alpha^2(\M)$ with $\|w_0\|=\eta_0$ such that
$$
\widetilde{\mathcal E}_{\lambda,\mu}(w_0)>\sup_{B_c(0,\varrho_0)}\widetilde{\mathcal E}_{\lambda,\mu}-\frac{1}{2}(\varrho_0^2-\eta^2_0),
$$
and the second claim is proved as well.
\end{proof}

We are now in a position to prove our existence result.

\emph{Proof of Theorem \ref{principale}.} Fix $\mu>0$ and let $\varrho_{\mu,\max}>0$ be the global maximizer of $l_\mu$. Set $\varrho_{0,\mu}:=\min\{\bar{\varrho}_\mu, \varrho_{\mu,\max}\}$, $\bar{\varrho}_\mu$ being defined by \eqref{rhosegnatomu}, and $\Lambda_\mu:=(0,l_\mu(\varrho_{0,\mu}))$.

Taking $\lambda\in\Lambda_\mu$, there exists $\varrho_{0,\mu,\lambda}\in(0,\varrho_{0,\mu})$ such that
\begin{equation}\label{rangelambda}
0<\lambda<\frac{\varrho_{0,\mu,\lambda}-\mu c_{2^*}^{2^*}\left\| K\right\|_\infty \varrho_{0,\mu,\lambda}^{2^*-1}}{a_1c_{2^*}Vol_g(\M)^\frac{2^*-1}{2^*} + a_2c_{2^*}^q Vol_g(\M)^\frac{2^*-q}{2^*}\varrho_{0,\mu,\lambda}^{q-1}}.	
\end{equation}

Since $\varrho_{0,\mu,\lambda}<\bar{\varrho}_\mu$, by Lemma \ref{lemmasemicontinuita} the functional $\mathcal E_{\lambda,\mu}$ is sequentially weakly lower semicontinuous in $B_c(0,\varrho_{0,\mu,\lambda})$ and so there exists $w_{0,\mu,\lambda}\in\linebreak B_c(0,\varrho_{0,\mu,\lambda})$ such that
$$
\mathcal E_{\lambda,\mu}(w_{0,\mu,\lambda})=\min_{B_c(0,\varrho_{0,\mu,\lambda})}\mathcal E_{\lambda,\mu}.
$$

Suppose by contradiction that $\|w_{0,\mu,\lambda}\|=\varrho_{0,\mu,\lambda}$. Fix $\varepsilon\in(0,\varrho_{0,\mu,\lambda})$ and define
$$
\varphi_{\lambda,\mu}(\varepsilon,\varrho_{0,\mu,\lambda}):=\frac{\displaystyle \sup_{B_c(0,\varrho_{0,\mu,\lambda})}\widetilde{\mathcal E}_{\lambda,\mu} - \sup_{B_c(0,\varrho_{0,\mu,\lambda}-\varepsilon)}\widetilde{\mathcal E}_{\lambda,\mu}}{\varepsilon}.
$$

With the aid of $(f_1)$ we get
\begin{align*}
&\quad\ \varphi_{\lambda,\mu}(\varepsilon,\varrho_{0,\mu,\lambda})\\
&\leq\frac{1}{\varepsilon}\sup_{w\in B_c(0,1)}\int_\M\left|\int_{(\varrho_{0,\mu,\lambda}-\varepsilon)w(\sigma)}^{\varrho_{0,\mu,\lambda} w(\sigma)}\left(\mu \left\| K\right\|_\infty |t|^{2^*-1} +\lambda |f(t)|\right) dt\right|d\sigma_g\\
&\leq \frac{1}{\varepsilon}\sup_{w\in B_c(0,1)}\int_\M\left|\int_{(\varrho_{0,\mu,\lambda}-\varepsilon)w(\sigma)}^{\varrho_{0,\mu,\lambda} w(\sigma)}\left(\mu \left\| K\right\|_\infty |t|^{2^*-1} + a_1\lambda + a_2\lambda |t|^{q-1}\right) dt\right|d\sigma_g\\
& \leq \frac{c_{2^*}^{2^*}\mu \left\| K\right\|_\infty }{2^*}\left(\frac{\varrho_{0,\mu,\lambda}^{2^*}-(\varrho_{0,\mu,\lambda}-\varepsilon)^{2^*}}{\varepsilon}\right) + a_1\lambda c_{2^*}Vol_g(\M)^\frac{2^*-1}{2^*} \\
&\quad + a_2 \lambda \frac{c_{2^*}^q}{q}Vol_g(\M)^\frac{2^*-q}{2^*}\left(\frac{\varrho_{0,\mu,\lambda}^q-(\varrho_{0,\mu,\lambda}-\varepsilon)^q}{\varepsilon}\right)
\end{align*}
and taking the limsup for $\varepsilon\rightarrow 0^+$ we get
\begin{align}\label{VGa1Glu1}
\limsup_{\varepsilon\rightarrow 0^+}\varphi_{\lambda,\mu}(\varepsilon,\varrho_{0,\mu,\lambda}) &\leq c_{2^*}^{2^*}\mu \left\| K\right\|_\infty \varrho_{0,\mu,\lambda}^{2^*-1} + \lambda a_1 c_{2^*}Vol_g(\M)^\frac{2^*-1}{2^*} \\
\notag &\quad + \lambda a_2 c_{2^*}^q Vol_g(\M)^\frac{2^*-q}{2^*}\varrho_{0,\mu,\lambda}^{q-1},
\end{align}
which, due to \eqref{rangelambda}, forces
$$
\limsup_{\varepsilon\rightarrow 0^+}\varphi_{\lambda,\mu}(\varepsilon,\varrho_{0,\mu,\lambda})<\varrho_{0,\mu,\lambda}.
$$
Therefore, invoking Lemma \ref{proprietaEtilde}, one has
\begin{equation*}\label{VGa1Glu}
\inf_{w\in B(0,\varrho_{0,\mu,\lambda})}\frac{\displaystyle \sup_{B_c(0,\varrho_{0,\mu,\lambda})}\widetilde{\mathcal E}_{\lambda,\mu} - \widetilde{\mathcal E}_{\lambda,\mu}(w)}{\varrho_{0,\mu,\lambda}^2-\left\|w\right\|^2}<\frac{1}{2}
\end{equation*}
and there exists $\bar w_{\mu,\lambda}\in B(0,\varrho_{0,\mu,\lambda})$ such that, for every $w\in B_c(0,\varrho_{0,\mu,\lambda})$,
$$
\widetilde{\mathcal E}_{\lambda,\mu}(w)\leq \sup_{B_c(0,\varrho_{0,\mu,\lambda})}\widetilde{\mathcal E}_{\lambda,\mu} < \widetilde{\mathcal E}_{\lambda,\mu}(\bar w_{\mu,\lambda}) + \frac{1}{2}(\varrho_{0,\mu,\lambda}^2-\|\bar w_{\mu,\lambda}\|^2),
$$
 which we can rewrite as
\begin{equation}\label{John}
\mathcal E_{\lambda,\mu}(\bar w_{\mu,\lambda}):=\frac{1}{2}\|\bar w_{\mu,\lambda}\|^2 - \widetilde{\mathcal E}_{\lambda,\mu}(\bar w_{\mu,\lambda})<\frac{\varrho_{0,\mu,\lambda}^2}{2}-\widetilde{\mathcal E}_{\lambda,\mu}(w).
\end{equation}
Evaluating the previous inequality at $w=w_0$, we deduce
\begin{equation*}
\mathcal E_{\lambda,\mu}(\bar w_{\mu,\lambda})< \frac{1}{2}\varrho_{0,\mu,\lambda}^2 - \widetilde{\mathcal E}_{\lambda,\mu}(w_{0,\mu,\lambda}) = \mathcal E_{\lambda,\mu}(w_{0,\mu,\lambda}),
\end{equation*}
against the minimality of $w_{0,\mu,\lambda}$. In conclusion, $w_{0,\mu,\lambda} \in B(0,\varrho_{0,\mu,\lambda})$ and is therefore a local minimum for $\mathcal E_{\lambda,\mu}$ and a solution to \eqref{problema}.

The final task is now to show that $w_{0,\mu,\lambda}$ is not identically $0$ on $\M$. To this end, fix a constant $a\in(0,+\infty)$. Thanks to $(f_2)$, for all $c>0$ there exists $\delta_c>0$ such that
$$
F(t)\geq c t^2 \quad \mbox{for any } t\in(0,\delta_c).
$$
So if $t\in(0,\delta_c/a)$, we obtain
\begin{align*}
\mathcal E_{\lambda,\mu}(ta) & = \frac{1}{2}a^2 Vol_g(\M) t^2 -\frac{\mu}{2^*}a^{2^*}\left\| K\right\|_1 t^{2^*} - \lambda\int_\M F(ta)d\sigma_g\\
& \leq  \left( \frac{1}{2} - \lambda c \right) a^2 Vol_g(\M) t^2 -\frac{\mu}{2^*}a^{2^*}\left\| K\right\|_1 t^{2^*}\\
&<0
\end{align*}
for big enough $c>0$. As a result, $0$ is not a local minimizer of $\mathcal E_{\lambda,\mu}$ and $w_{0,\mu,\lambda}$ is not the null function.
\qed

\begin{remark}
\rm{It is worth noticing that assumption $(f_2)$ comes into play only at the end of the proof of Theorem \ref{principale}, to prevent $0$ from being a local minimum point of the energy. Therefore it can be replaced by any other assumption, compatible with $(f_1)$, which ensures this fact. For instance, if $f(0)\neq 0$, then it is easy to verify that $0$ is not a solution to \eqref{problema} and so $w_{0,\lambda,\mu}$ is nonzero. Another hypothesis, more restrictive than $(f_2)$, which serves our purpose is
$$
\liminf_{t\to 0^+}\frac{F(t)}{t^b}>0, \quad \text{for some } b\in(1,2).
$$
}
\end{remark}

The next result shows that, by lightly strengthening the assumptions on $f$, one can determine the sign of the local minimizer.

\begin{theorem}\label{corollarypm0}
Let $f:\R\to\R$ satisfy $(f_1)$, $(f_2)$ and $f(0)=0$. Then for every $\mu>0$ and sufficiently small $\lambda$, problem \eqref{problema} admits a nonnegative solution $w^\star_{0,\mu,\lambda}$ in $H_\alpha^2(\M)$.
\end{theorem}

\begin{proof}
Associated with the function
$$
F_+(t):=\int_0^t  f_+(\tau)d\tau,
$$
for every $t\in\R$, where
$$
f_+(\tau):=	
\begin{cases}
f(\tau) & \mbox{ if } \tau>0\\
0 & \mbox{ if } \tau\leq 0,
\end{cases}
$$
let us introduce the functional $\mathcal E_{\lambda,\mu}^{+}: H_\alpha^2(\M) \to \R$ given by
\begin{equation}\label{Funzionale+}
\mathcal E_{\lambda,\mu}^{+}(w) :=\frac 1 2 \left\| w\right\|^2
	-\frac{\mu}{2^*}\int_{\mathcal M} K(\sigma)w_{+}(\sigma)^{2^*} d\sigma_g -\lambda\int_{\mathcal M} F_+(w(\sigma))d\sigma_g,
\end{equation}
where $u_+:=\max\{u,0 \}$, for any $w\in H_\alpha^2(\M)$.\par
 It is a simple matter to check that $\mathcal E_{\lambda,\mu}^{+}$ is well-defined and G\^{a}teaux-differentiable on $H_\alpha^2(\M)$, so by Theorem \ref{principale}, for every $\mu>0$ and $\lambda$ sufficiently small, it admits a critical point $w^\star_{0,\mu,\lambda}\in H_\alpha^2(\M)$.	

We claim that $w^\star_{0,\mu,\lambda}$ is nonnegative on $\M$. Indeed, since $w^\star_{0,\mu,\lambda}\!\in\!
H_\alpha^2(\M)$, it follows that $(w^\star_{0,\mu,\lambda})_-:=\max\{-w^\star_{0,\mu,\lambda}, 0\}$ belongs to $H_\alpha^2(\M)$ as
well. So, taking also into account the relationship
	\begin{align*}
	\left\langle w^\star_{0,\mu,\lambda},(w^\star_{0,\mu,\lambda})_-\right\rangle & = \int_{\M}\left\langle \nabla
w^\star_{0,\mu,\lambda}(\sigma),\nabla (w^\star_{0,\mu,\lambda})_-(\sigma) \right\rangle_g  d\sigma_g\\
	&\quad  +  \int_{\M} \alpha(\sigma)\left\langle w^\star_{0,\mu,\lambda}(\sigma),
(w^\star_{0,\mu,\lambda})_-(\sigma)\right\rangle_g d\sigma_g \\
	& = - \int_{\M} \left( |\nabla (w^\star_{0,\mu,\lambda})_- (\sigma)|^2 +
\alpha(\sigma)(w^\star_{0,\mu,\lambda})_-(\sigma)^2\right)  d\sigma_g,
	\end{align*}
	 we get

\begin{align*}
	- \left\|(w^\star_{0,\mu,\lambda})_-\right\|^2& = \left\langle
w^\star_{0,\mu,\lambda},(w^\star_{0,\mu,\lambda})_-\right\rangle\\
	& = \mu\int_{\mathcal M} K(\sigma) (w^\star_{0,\mu,\lambda})_{+}(\sigma)^{2^*-1}(w^\star_{0,\mu,\lambda})_{-}(\sigma)
d\sigma_g\\
&\quad +\lambda\int_{\mathcal M} f_+(w^\star_{0,\mu,\lambda}(\sigma))(w^\star_{0,\mu,\lambda})_-(\sigma) d\sigma_g\\
	& = \lambda\int_{\mathcal M} f_+(w^\star_{0,\mu,\lambda}(\sigma))(w^\star_{0,\mu,\lambda})_-(\sigma) d\sigma_g\\
	& = 0.
   \end{align*}
As a result, $\|(w^\star_{0,\mu,\lambda})_-\|=0$ and hence $w^\star_{0,\mu,\lambda}\geq 0$ a.e. on $\mathcal M$.
\end{proof}

\begin{remark}
\rm{In the same way, if $f:\R\to\R$ satisfies $(f_1)$,
		\begin{itemize}
			\item[$(f_2')$ ] $\displaystyle\liminf_{t\to 0^-}\frac{F(t)}{t^2} =+\infty$,
			\end{itemize}
and $f(0)=0$, then, arguing as in Theorem \ref{corollarypm0}, one can study the existence of a nonpositive solution to \eqref{problema}. It suffices to consider the functional
\begin{equation}\label{Funzionale-}
\mathcal E_{\lambda,\mu}^{-}(w) :=\frac 1 2 \left\| w\right\|^2
-\frac{\mu}{2^*}\int_{\mathcal M} K(\sigma)w_-(\sigma)^{2^*} d\sigma_g -\lambda\int_{\mathcal M} F_-(w(\sigma))d\sigma_g,
\end{equation}
where
	$$
	F_-(t):=\int_0^t  f_-(\tau)d\tau,
	$$
for every $t\in\R$, and
	$$
	f_-(\tau):=	
	\begin{cases}
	f(\tau) & \mbox{ if } \tau<0\\
	0 & \mbox{ if } \tau\geq 0.
	\end{cases}
	$$
}
\end{remark}

The underlying idea of the proof of Theorem \ref{principale} remains valid when adding a term singular at zero, i.e. to treat the following singular variant of problem \eqref{problema}:
\begin{equation}\tag{$P_{\lambda,\mu}^\star$}\label{problemasingolare}
\begin{cases}
	-\Delta_g w + \alpha(\sigma)w  = \mu K(\sigma) w^\frac{d+2}{d-2} +\lambda \left( w^{r-1} + f(w)\right),  \quad \sigma\in\M \\
w\in H^2_\alpha(\M), \quad  w>0 \mbox{ in } \mathcal M,
\end{cases}
\end{equation}
where $r\in(0,1)$ and $f:[0,+\infty)\to[0,+\infty)$ is continuous and subcritical. In this context a weak solution to \eqref{problemasingolare} is meant to be any $w\in H_\alpha^2(\M)$ such that $w>0$ a.e. in $\mathcal M$, $w^{r-1}z\in L^1(\mathcal M)$ for any $z\in H_\alpha^2(\M)$ and
$$
\left\langle w,z\right\rangle - \mu\int_\mathcal M K(\sigma) w^\frac{d+2}{d-2}z d\sigma_g - \lambda\int_\mathcal M \left( w^{r-1} + f(w)\right) z d\sigma_g =0
$$
for each $z\in H_\alpha^2(\M)$. The energy naturally associated with \eqref{problemasingolare} is
\begin{align}\label{energiasingolare}
	\mathcal E_{\lambda,\mu}(w)&:=\frac{1}{2}\left\|w\right\|^2 -\frac{\mu}{2^*}\int_\mathcal M K(\sigma)(w^+)^{2^*}d\sigma_g\\
	\notag &\quad -\frac{\lambda}{r}\int_\mathcal M (w^+)^r d\sigma_g -\lambda\int_\mathcal M F(w^+) d\sigma_g,
\end{align}
for all $w\in H_\alpha^2(\M)$. As before, for any positive $\mu$, the range of $\lambda$ for which \eqref{problemasingolare} admits nontrivial solutions is strictly related to the maximum of an auxiliary rational function similar to \eqref{funzionelmu}.

\begin{theorem}\label{princsingolare}
	Let $f:[0,+\infty)\to[0,+\infty)$ be a continuous function for which
	\begin{itemize}
		\item[$(f_1')$] there exist $a_1,a_2\geq 0$ and $q\in[1,2^*)$ such that
		$$
		f(t)\leq a_1 + a_2 t^{q-1} \quad  \mbox{ for all } t\geq 0.
		$$
	\end{itemize}
	For any $\mu>0$ let $m_\mu:[0,+\infty)\to\R$ be the function defined by
	$$
	m_\mu (t):=\frac{t^{2-r}-\mu c_{2^*}^{2^*}\left\| K\right\|_\infty t^{2^*-r}}{ c_{2^*}^r Vol_g(\M)^\frac{2^*-r}{2^*} \!+\! a_1 c_{2^*}Vol_g(\M)^\frac{2^*-1}{2^*} t^{1-r} \!+\! a_2 c_{2^*}^q Vol_g(\M)^\frac{2^*-q}{2^*} t^{q-r}},
	$$
for avery  $t\geq 0$.\par
	Then for every $\mu>0$ there exists an open interval
	$$
	\Lambda _\mu \subseteq \left(0,\max_{[0,+\infty)}m_\mu\right)
	$$
	such that, for every $\lambda\in\Lambda_\mu$, \eqref{problemasingolare} admits a nontrivial weak solution $\tilde{w}_{0,\lambda,\mu}\in H_\alpha^2(\M)$.
\end{theorem}

\begin{proof}
	It is clear that for all $\lambda,\mu>0$ the functionals
	\begin{align*}
		H_\alpha^2(\M)\ni w &\mapsto\frac{1}{2}\left\|w\right\|^2 -\frac{\mu}{2^*}\int_\M K(\sigma)(w^+)^{2^*} d\sigma_g,\\
			H_\alpha^2(\M)\ni w &\mapsto \frac{\mu}{2^*}\int_\M K(\sigma) (w^+)^{2^*} d\sigma_g +\frac{\lambda}{r}\int_\M (w^+)^r d\sigma_g +\lambda\int_\M F(w^+) d\sigma_g
	\end{align*}
	satisfy	Lemmas \ref{lemmasemicontinuita} and \ref{proprietaEtilde}, respectively.

	Arguing exactly as in Theorem \ref{principale}, we deduce that the functional $\mathcal E_{\lambda,\mu}$ defined by \eqref{energiasingolare}, attains a minimum $\tilde{w}_{0,\lambda,\mu}$ on a sufficiently small ball $$B_c(0,\varrho_{0,\mu,\lambda})\subset H_\alpha^2(\M)$$ and such a minimum is not identically zero. Indeed, fixing $w\in H_\alpha^2(\M)$, $w>0$ on $\M$, if $t>0$ one has
	\begin{align*}
	\mathcal E_{\lambda,\mu}(tw) &\leq \frac{1}{2}\left\| w\right\|^2 t^2 - \frac{\mu}{2^*}\essinf_\M K \left\| w\right\|_{2^*}^{2^*}t^{2^*} +\lambda a_1\left\|w\right\|_{1} t \\
	& \quad +\frac{\lambda a_2}{q}\left\| w\right\|_{q}^q t^q -\frac{\lambda}{r}\int_\M |w|^r d\sigma_g \; t^r
	\end{align*}
	and hence $\mathcal E_{\lambda,\mu}(tw)$ is negative for all small enough $t$. Arguments similar to those of  \cite[Theorem~4.1]{farfar2015a} finally show that $\tilde{w}_{0,\lambda,\mu}$ weakly solves \eqref{problemasingolare}.
   \end{proof}
   
\section{Existence of MP-type solutions}\label{mountainpass}
In this section we establish another existence result for \eqref{problema}, this time of the mountain pass type. The abstract tool we rely upon is a version of the mountain pass theorem without the Palais-Smale condition (\cite[Theorem~2.2]{brenir1983positive}) and, as usual in the treatment of critical problems via this approach, the crucial point is the relationship between the mountain pass level and the constant $S$ defined by \eqref{costanteS}.

\begin{theorem}\label{secondoteorema}
Let $f:\R\to\R$ be a continuous function satisfying $(f_1)$ and
\begin{itemize}
\item[$(f_3)$] $\displaystyle\lim_{t\to 0}\frac{f(t)}{t}=0$,
\end{itemize}
and let $\lambda,\mu>0$.
Furthermore, assume that there exists $w^\star\in H_\alpha^2(\M)\setminus\{0\}$, $w^\star\geq 0$ a.e. in $\M$, such that
\begin{itemize}
\item[$(H_1)$] $\displaystyle\sup_{\tau\geq 0}\mathcal E_{\lambda,\mu}(\tau w^\star)<\frac{S^\frac{d}{2}}{d(\mu\left\| K\right\|_\infty)^{\frac{d-2}{2}}}$.	
\end{itemize}
Then problem \eqref{problema} admits a nontrivial solution.	
\end{theorem}

\begin{proof}
Let us show first that $\mathcal E_{\lambda,\mu}$ possesses the geometry required by Theorem  2.2 of \cite{brenir1983positive}.

\begin{claim}\label{claimone}
There exist $\varrho,\beta>0$ such that for any $w\in \partial B(0,\varrho)$ one has $\mathcal E_{\lambda,\mu}(w)\geq\beta$.
\end{claim}

Fix $\varepsilon\in \left( 0,\lambda^{-1} c_2^{-2}\right) $. Thanks to $(f_3)$, there exists $\delta_\epsilon>0$ so that
\begin{equation}
|f(t)| \leq \varepsilon |t|
\end{equation}
for every $t\in(-\delta_\varepsilon, \delta_\varepsilon)$. On the other hand, $(f_1)$ allows us to deduce that
$$
\lim_{|t|\to +\infty}\frac{|f(t)|}{|t|^{2^*-1}}=0
$$
and therefore there exists $a_\varepsilon>0$ such that
\begin{equation}\label{maggiorazioneg}
|f(t)|\leq \varepsilon |t|^{2^*-1} + a_\varepsilon
\end{equation}
for all $t\in\R$ and, consequently,
\begin{equation}
|f(t)| \leq \varepsilon |t|^{2^*-1} + a_\varepsilon\frac{|t|^{2^*-1}}{\delta_\varepsilon^{2^*-1}} = b_\varepsilon |t|^{2^*-1}
\end{equation}
for every $t\in\R\setminus(-\delta_\varepsilon,\delta_\varepsilon)$. Hence we get
\begin{equation}
|F(t)|\leq \frac{\varepsilon}{2} t^2 + \frac{b_\varepsilon}{2^*}|t|^{2^*}
\end{equation}
for every $t\in\R$. In the light of the above estimates, for every $w\in H_\alpha^2(\M)$ we obtain
\begin{align*}
\mathcal E_{\lambda,\mu}(w) & \geq \frac{1}{2}\left\| w\right\|^2 -\frac{\mu}{2^*} \left\| K\right\|_\infty \left\| w\right\|_{2^*}^{2^*}  -\frac{\lambda\varepsilon}{2}\left\| w\right\|_2^2 -\frac{\lambda b_\varepsilon}{2^*}  \left\| w\right\|_{2^*}^{2^*} \\
& \geq \left(\frac{1-\lambda\varepsilon c_2^2}{2} \right)\left\| w\right\|^2 -\left(\frac{\mu\left\| K\right\|_\infty+\lambda b_\varepsilon}{2^*}\right)c_{2^*}^{2^*}\left\| w\right\|^{2^*}.
   \end{align*}
Thanks to the choice of $\varepsilon$, for suitable $a,b>0$ we have
$$
\mathcal E_{\lambda,\mu}(w) \geq a \left\| w\right\|^2 \left(1-b\left\|w\right\|^\frac{4}{d-2}  \right)
$$
and so, by choosing $\varrho\in \left( 0, b^{\frac{2-d}{4}} \right) $, we get
$$
\inf_{\partial B(0,\varrho)} \mathcal E \geq a \varrho^2 \left(1-b\varrho^\frac{4}{d-2}\right)=:\beta>0
$$
and the conclusion is achieved.

\begin{claim}
There exists $w_1\in H_\alpha^2(\M)$ such that $w_1\geq 0$ a.e. in $\M$, $\|w_1\|>\varrho$ and $\mathcal E_{\lambda,\mu}(w_1)<\beta$, where $\varrho,\beta$ have the same meaning as in Claim \ref{claimone}.	
\end{claim}

If $\tau>0$ and $w^\star$ is the function defined in $(H_1)$, we get
\begin{align*}
\mathcal E_{\lambda,\mu}\left( \frac{\tau w^\star}{\left\| w^\star\right\| }\right)  & = \frac{1}{2}\tau^2 \!-\!\frac{\mu}{2^*\left\|w^\star\right\|^{2^*}}\int_\M K(\sigma)|\tau w^\star|^{2^*} d\sigma_g \!-\!\lambda \int_\M F\left(  \frac{\tau w^\star}{\left\| w^\star\right\|}\right)  d\sigma_g\\
& \leq \frac{1}{2}\tau^2 - \frac{\mu}{2^*\left\| w^\star\right\| ^{2^*}}\essinf_\M K \left\| w^\star\right\|_{2^*}^{2^*} \tau^{2^*}\\
&\quad + \lambda a_1 \frac{\left\| w^\star\right\|_1}{\left\| w^\star\right\|} \tau + \frac{\lambda a_2}{q}\cdot\frac{\left\| w^\star\right\|_q^q}{\left\| w^\star\right\|^q} \tau^q
\end{align*}
and therefore $\mathcal E_{\lambda,\mu}\left(\tau w^\star/\left\| w^\star\right\|\right)\to -\infty$ as $\tau\to +\infty$. So it suffices to pick $w_1:= \tau_0 w^\star/\left\| w^\star\right\|$, with large enough $\tau_0>0$, to obtain the claim.

\bigskip

Now, set
\begin{equation}\label{quotaMP}
c:=\inf_{\gamma\in \Gamma}\sup_{t\in[0,1]}\mathcal E_{\lambda,\mu}(\gamma(t)),
\end{equation}
where
$$
\Gamma:=\{\gamma\in C^0([0,1], H_\alpha^2(\M)) : \gamma(0)=0 \mbox{ and } \gamma(1)=w_1\}.
$$
For any $\gamma\in\Gamma$, the function $t\mapsto \left\| \gamma(t)\right\| $ is continuous on $[0,1]$, so by the intermediate value theorem there exists $\bar t\in(0,1)$ such that $\left\| \gamma(\bar t)\right\| =\varrho$. As a result
$$
\sup_{t\in[0,1]}\mathcal E_{\lambda,\mu}(\gamma(t)) \geq \mathcal E_{\lambda,\mu}(\gamma(\bar t)) \geq \inf_{\partial B(0,\varrho)} \mathcal E_{\lambda,\mu}
$$
which implies $c\geq\beta$. Moreover, thanks to $(H_1)$ and the fact that the map $t\mapsto  w_1 t$, $t\in[0,1]$, is an element of $\Gamma$, one has the estimate
\begin{equation}\label{relazionec}
c  \leq \sup_{\tau\geq 0}\mathcal E_{\lambda,\mu}(\tau w_1)<\frac{S^\frac{d}{2}}{d(\mu\left\| K\right\|_\infty)^{\frac{d-2}{2}}}.
\end{equation}

Next, due to Theorem 2.2 of \cite{brenir1983positive}, there exists a sequence $\{w_j\}_{j\in\N}\subset H_\alpha^2(\M)$ satisfying
\begin{equation}\label{quasipuntocritico}
\mathcal E_{\lambda,\mu}(w_j)\to c, \quad \mathcal E'_{\lambda,\mu}(w_j)\to 0, \quad \mbox{as } j\to\infty.
\end{equation}

\begin{claim}
	The sequence $\{w_j\}_{j\in\N}$ is bounded in $H_\alpha^2(\M)$.
\end{claim}

For large enough $j$, using also \eqref{maggiorazioneg} with $\varepsilon\in\left(0,\mu\essinf_\M K/\lambda(d-1) \right)$, we have
\begin{align*}
c + 1 + \left\| w_j\right\| & \geq \mathcal E_{\lambda,\mu}(w_j) - \frac{1}{2}\mathcal E'_{\lambda,\mu}(w_j)(w_j) \\
& = \frac{\mu}{d} \int_\M K(\sigma)|w_j|^{2^*}d\sigma_g -\lambda \int_\M F(w_j) d\sigma_g + \frac{\lambda}{2}\int_\M f(w_j) w_j d\sigma_g\\
& \geq \frac{\mu}{d}\essinf_\M K \left\| w_j\right\|_{2^*}^{2^*} -\left(\frac{\lambda\varepsilon}{2^*}+\frac{\lambda\varepsilon}{2} \right) \left\| w_j\right\|_{2^*}^{2^*} -\frac{3}{2}\lambda a_\varepsilon \left\|  w_j\right\|_1\\
& \geq \left(\frac{2\mu \essinf_\M K -\lambda\varepsilon(d-2)-\lambda\varepsilon d}{2d} \right) \left\| w_j\right\|_{2^*}^{2^*} - \frac{3}{2}\lambda a_\varepsilon c_1 \left\|  w_j\right\|,
\end{align*}
and therefore,  for some $\kappa_1>0$,
\begin{equation}\label{kappa1}
\left\| w_j\right\|_{2^*}^{2^*} \leq \kappa_1(1+\left\| w_j\right\| ).
\end{equation}
On the other hand, again for large values of $j$, taking \eqref{kappa1} into account one has
\begin{align*}
	c+1 &\geq \mathcal E_{\lambda,\mu}(w_j) \geq \frac{1}{2} \left\| w_j\right\|^2 - \left( \frac{\mu}{2^*}\left\| K\right\|_\infty + \frac{\lambda\varepsilon}{2^*}\right)  \left\| w_j\right\|_{2^*}^{2^*} -\lambda a_\varepsilon c_1 \left\| w_j\right\| \\
& \geq \frac{1}{2} \left\| w_j\right\|^2 - \left(\frac{\mu\left\| K\right\|_\infty + \lambda\varepsilon (d-2) }{2d} \right)\kappa_1(1+\left\| w_j\right\| ) - \lambda a_\varepsilon c_1 \left\| w_j\right\|,
\end{align*}
which forces
$$
\left\| w_j\right\|^2 \leq \kappa_2(1+\left\| w_j\right\| )
$$
for some $\kappa_2>0$, and hence the boundedness of $\{w_j\}_{j\in\N}$.	

\medskip

As a consequence of the previous claim, there exists $w_\infty \in H_\alpha^2(\M)$ such that, up to a subsequence (still denoted by $\{w_j\}_{j\in\N}$), $w_j \rightharpoonup w_\infty$, that is
\begin{equation}\label{convergenze0}
\left\langle w_j,z\right\rangle  \to \left\langle w_\infty,z\right\rangle \quad \mbox{as } j\to\infty
\end{equation}
for any $z\in H_\alpha^2(\M)$. Moreover, since $\{w_j\}_{j\in\N}$ is bounded in $L^{2^*}(\M)$ as well, passing to a further subsequence, we get the validity of the following convergences:
\begin{equation}\label{convergenze}
\begin{split}
& w_j \rightharpoonup w_\infty \quad \mbox{ in } L^{2^*}(\M)\\
& w_j \to w_\infty \quad \mbox{ in } L^p(\M), \quad p\in[1,2^*) \\
& w_j \to w_\infty \quad \mbox{ a.e. in } \M\\
& |w_j|^\frac{4}{d-2}w_j \rightharpoonup  |w_\infty|^\frac{4}{d-2}w_\infty \quad \mbox{ in } L^\frac{2d}{d+2}(\M),
\end{split}
\end{equation}
as $j\to\infty$. Moreover, thanks to \eqref{maggiorazioneg} and \eqref{convergenze}, we easily obtain
\begin{align*}\label{convfae}
& f(w_j(\cdot)) \to  f(w_\infty(\cdot)) \quad \mbox{ a.e. in } \M\\
& f(w_j(\cdot)) \rightharpoonup  f(w_\infty(\cdot)) \quad \mbox{ in }  L^\frac{2d}{d+2}(\M),
\end{align*}
as $j\to\infty$. As a result,
$$
\int_\M f(w_j)z d\sigma_g \to  \int_\M f(w_\infty)z d\sigma_g \quad \mbox{as } j\to\infty
$$
for any $z\in L^{2^*}(\M)$ and, a fortiori,
\begin{equation}\label{convffi1}
\int_\M f(w_j)z d\sigma_g \to  \int_\M f(w_\infty)z d\sigma_g \quad \mbox{as } j\to\infty
\end{equation}
for any $z\in H_\alpha^2(\M)$. Now, recalling that
$$
\mathcal E'_{\lambda,\mu}(w_j)(z) = \left\langle w_j,z\right\rangle  - \mu\int_\M K(\sigma) |w_j|^\frac{4}{d-2}w_j z d\sigma_g - \lambda \int_\M f(w_j) z d\sigma_g,
$$
for any $z\in H_\alpha^2$, passing to the limit as $j\to\infty$ in the above equality and taking  \eqref{quasipuntocritico}, \eqref{convergenze0}, \eqref{convergenze} and \eqref{convffi1} and the fact that $K\in\Lambda_+(\M)$ into account, we get
$$
\left\langle w_\infty,z\right\rangle - \mu\int_\M K(\sigma)|w_\infty|^\frac{4}{d-2}w_\infty z d\sigma_g - \lambda \int_\M f(w_\infty) z d\sigma_g = 0
$$
for any $z \in H_\alpha^2(\M)$, so $w_\infty$ is the weak solution to \eqref{problema} we were looking for.

Arguing by contradiction, finally suppose that $w_\infty \equiv 0$ in $\M$. Then, due to \eqref{maggiorazioneg}, for any $j\in\N$ and for some $\kappa>0$ we would obtain
$$
\left| \int_\M f(w_j) w_j d\sigma_g\right|  \leq \varepsilon\left\| w_j\right\| _{2^*}^{2^*} + a_\varepsilon \left\| w_j\right\| _1 \leq \varepsilon \kappa + a_\varepsilon \left\| w_j\right\|_1,
$$
and
$$
\left| \int_\M F(w_j) d\sigma_g\right|   \leq \frac{\varepsilon}{2^*}\left\| w_j\right\|_{2^*}^{2^*} + a_\varepsilon \left\| w_j\right\|_1 \leq \frac{\varepsilon \kappa}{2^*} + a_\varepsilon \left\| w_j\right\|_1.
$$
Exploiting \eqref{convergenze} and taking $\limsup$ as $j\to\infty$ and $\lim$ as $\varepsilon\to 0$ in the above inequalities one has
\begin{equation}\label{fujlim}
\lim_{j\to\infty}\int_\M f(w_j) w_j d\sigma_g = \lim_{j\to\infty}\int_\M F(w_j) d\sigma_g = 0
\end{equation}
and therefore, since $\mathcal E'_{\lambda,\mu}(w_j)( w_j) \to 0$ as $j\to\infty$, 
\begin{equation}\label{nonso}
\left\| w_j\right\|^2  - \mu\int_\M K(\sigma) |w_j|^{2^*} d\sigma_g \to 0 \quad \mbox{as } j\to\infty.
\end{equation}
Now, the boundedness of $\{\|w_j\|\}_{j\in\N}$ in $\R$ implies that, up to a subsequence, there exists $L\in[0,+\infty)$ such that
\begin{equation}\label{L1}
\|w_j\|^2 \to L
\end{equation}
and so
\begin{equation}\label{L2}
\int_\M K(\sigma)|w_j|^{2^*} d\sigma_g\to \frac{L}{\mu}
\end{equation}
as $j\to +\infty$. Recalling that
\begin{align*}
	\mathcal E_{\lambda,\mu}(w_j)&=\frac 1 2\left\| w_j\right\|^2 -\frac{\mu}{2^*} \int_\M K(\sigma)|w_j|^{2^*} d\sigma_g\\
	&\quad -
\lambda\int_\M F(w_j)d\sigma_g \to c \quad \mbox{as } j\to\infty,
\end{align*}
it follows on account of \eqref{fujlim}, \eqref{L1} and \eqref{L2}, that
\begin{equation}\label{c}
c=\left( \frac 1 2 -\frac{1}{2^*}\right) L=\frac{L}{d},
\end{equation}
which due to $c\geq \beta>0$, forces $L>0$. On the other hand, from \eqref{L2} we get
\begin{equation}\label{nuova}
\lim_{j\to\infty} \int_\M |w_j|^{2^*} d\sigma_g \geq \frac{L}{\mu \left\| K\right\|_\infty}
\end{equation}
and moreover
$$
\left\| w_j\right\|^2 \geq S \|w_j\|_{2^*}^2
$$
so that, passing to the limit as $j\to\infty$ in the above relationship and taking  \eqref{L1} and \eqref{nuova} into account, we get
$$
L\geq S \left(\frac{L}{\mu\left\| K\right\|_\infty}\right)^\frac{2}{2^*},
$$
which combined with \eqref{c}, gives
$$
c\geq \frac{S^\frac{2^*}{2^*-2}}{d (\mu\left\| K\right\| _\infty)^{\frac{2}{2^*-2}}} = \frac{S^\frac{d}{2}}{d (\mu\left\| K\right\| _\infty)^{\frac{d-2}{2}}}.
$$
This contradicts \eqref{relazionec} and therefore $w_\infty\not\equiv 0$ in $\M$. The proof is hence complete.
\end{proof}

\section{The case of the sphere: applications to the Emden-Fowler equation}\label{SectionEF}
As already explained in the introduction, an interesting case of problem \eqref{problema} is
\begin{align}\tag{$\widetilde P_{\lambda,\mu}$}\label{problematilde}
	&\quad -\Delta_h w + s(1-s-d)w \\
	\notag &\qquad=\mu K(\sigma)|w|^\frac{4}{d-2}w + \lambda f(w), \quad \sigma\in\Sp^d, \; w\in H^2_1(\Sp^d)
\end{align}
where $\Sp^d$ is the unit sphere in $\R^{d+1}$, $h$ is the standard metric induced by the embedding $\Sp^d\hookrightarrow\R^{d+1}$, $s\in\R$ is a constant related to $d$ by the relationship $1-d<s<0$, and $\Delta_h$ is the Laplace-Beltrami operator on $(\Sp^d,h)$.

Existence results for \eqref{problematilde}, via a suitable change of coordinates, produce the existence of solutions to the following Emden-Fowler equation
\begin{align}\tag{$\overline P_{\lambda,\mu}$}\label{emdenfowler}
	-\Delta u  &=\mu |x|^\frac{2(2-2s-d)}{d-2}K\left(\frac{x}{|x|}\right)|u|^\frac{4}{d-2} u\\ \notag &\quad + \lambda |x|^{s-2} f\left(\frac{u}{|x|^s}\right) , \quad x\in \R^{d+1}\setminus\{0\}.
\end{align}
Indeed, let us seek solutions to \eqref{emdenfowler} of the form
\begin{equation}\label{coordsfer}
u(x)=r^s w(\sigma),
\end{equation}
where $(r,\sigma):=(|x|,x/|x|)\in(0,+\infty)\times\Sp^d$ are the spherical coordinates in $\R^{d+1}\setminus\{0\}$ and $w$ a smooth function defined on $\Sp^d$. Via \eqref{coordsfer} and taking into account that
\begin{align*}
\Delta u&=r^{-d}\frac{\partial}{\partial r}
\left(r^{d}\frac{\partial}{\partial r}(r^sw)\right)+r^{s-2}\Delta_{h}w\\
&=\big[s(d+s-1)+\Delta_{h}w\big]r^{s-2},
\end{align*}
it is easily seen that the study of problem \eqref{problematilde} can be successfully applied to treat problem \eqref{emdenfowler}.

\begin{theorem}
Let $d,s\in\R$ be such that $1-d<s<0$, $K\in\Lambda_+(\Sp^d)$ and $f:\R\to\R$ be a locally Lipschitz continuous function satisfying $(f_1)-(f_2)$. Furthermore, for any $\mu>0$, let $l_\mu:[0,+\infty)\to\R$ be the function defined by
\begin{equation}\label{funzionehmu}
	l_\mu(t):=\frac{t-\mu c_{2^*}^{2^*}\left\| K\right\|_\infty  t^{2^*-1}}{a_1 c_{2^*}\omega_d^\frac{2^*-1}{2^*} + a_2 c_{2^*}^q \omega_d^\frac{2^*-q}{2^*} t^{q-1}} \quad  \mbox{ for all } t\geq 0.
\end{equation}
	Then for every $\mu>0$ there exists an open interval
	$$
	\Lambda _\mu \subseteq \left(0,\max_{[0,+\infty)}l_\mu\right)
	$$
	such that for every $\lambda\in\Lambda_\mu$, the problem
	\begin{align}\tag{$\overline P_{\lambda,\mu}$}\label{emdenfowler}
		-\Delta u  &=\mu |x|^\frac{2(2-2s-d)}{d-2}K\left(\frac{x}{|x|}\right)|u|^\frac{4}{d-2} u \\ \notag &\quad + \lambda |x|^{s-2} f\left(\frac{u}{|x|^s}\right) , \quad x\in \R^{d+1}\setminus\{0\}
	\end{align} admits a nontrivial solution.
	\end{theorem}

\begin{proof}
Let us appeal to Theorem \ref{principale} by choosing $(\M,g)=(\Sp^d,h)$ and $\alpha(\sigma):=s(1-s-d)$ for every $\sigma\in\Sp^d$. Thanks to the relationship between $d$ and $s$, clearly $\alpha\in \Lambda_+(\Sp^d)$. So the problem
\begin{align}
	&\quad -\Delta_h w + s(1-s-d)w \\
	\notag &\qquad=\mu K(\sigma)|w|^\frac{4}{d-2}w + \lambda f(w), \quad \sigma\in\Sp^d, \; w\in H^2_1(\Sp^d)
\end{align}
admits at least a nontrivial solution $w_{0,\mu,\lambda}\in H_\alpha^2(\Sp^d)$.\par
 But, due to \eqref{coordsfer}, $u(x)=|x|^s w_{0,\mu,\lambda}(x/|x|)$ is a nontrivial solution to \eqref{emdenfowler} and the proof is completed.
\end{proof}

In the same manner, we can obtain the following 

\begin{theorem}
	Let $d,s\in\R$ be such that $1-d<s<0$, $K\in\Lambda_+(\Sp^d)$ and $f:\R\to\R$ be a locally Lipschitz continuous function satisfying $(f_1)-(f_2)$ and $f(0)=0$. Then for every $\mu>0$ and sufficiently small $\lambda$, problem \eqref{emdenfowler} admits a nonnegative solution.
\end{theorem}

and

\begin{theorem}\label{Th11}
Let $d,s\in\R$ be such that $1-d<s<0$, $K\in\Lambda_+(\Sp^d)$.
Furthermore, let $r\in(0,1)$ and $f:[0,+\infty)\to[0,+\infty)$ be a locally Lipschitz continuous function for which
\begin{itemize}
	\item[$(f_1')$] there exist $a_1,a_2\geq 0$ and $q\in[1,2^*)$ such that
	$$
	f(t)\leq a_1 + a_2 t^{q-1} \quad  \mbox{ for all } t\geq 0.
	$$
\end{itemize}
For any $\mu>0$ let $m_\mu:[0,+\infty)\to\R$ be the function defined by
$$
m_\mu (t):=\frac{t^{2-r}-\mu c_{2^*}^{2^*}\left\| K\right\|_\infty t^{2^*-r}}{ c_{2^*}^r \omega_d^\frac{2^*-r}{2^*} + a_1 c_{2^*}\omega_d^\frac{2^*-1}{2^*} t^{1-r} + a_2 c_{2^*}^q \omega_d^\frac{2^*-q}{2^*} t^{q-r}} \quad  \mbox{ for all } t\geq 0.
$$
Then for every $\mu>0$ there exists an open interval
$$
\Lambda _\mu \subseteq \left(0,\max_{[0,+\infty)}m_\mu\right)
$$
such that, for every $\lambda\in\Lambda_\mu$, the problem
\begin{equation*}
\begin{cases}
	-\Delta_h w + s(1-s-d)w  = \mu K(\sigma) w^\frac{d+2}{d-2} +\lambda \left( w^{r-1} + f(w)\right),  \quad \sigma\in \Sp^d \\[1ex]
w\in H^2_1(\Sp^d), \quad  w>0 \mbox{ in } \Sp^d,
\end{cases}
\end{equation*}
admits a solution.
\end{theorem}

\section*{Acknowledgements} 
The manuscript was realized under the auspices of the Italian MIUR project \textit{Variational methods, with applications to problems in mathematical physics and geometry} (2015KB9WPT 009). The second author was supported by the Slovenian Research Agency grants P1-0292, N1-0114, N1-0083, N1-0064, and J1-8131. The third author was partially supported by the INdAM-GNAMPA Project 2017 \textit{Metodi variazionali per fenomeni non-locali}.
 The authors  thank the referees for comments and suggestions.

\end{document}